\documentclass[12pt]{amsart}
\usepackage{graphicx}
\usepackage{amssymb, amsmath, amsthm, amsfonts, enumerate}
\usepackage{amsmath,setspace,scalefnt}
\usepackage{verbatim}
\usepackage{multirow}
\usepackage{mathtools}
\usepackage{float}
\usepackage{enumitem}
\restylefloat{figure}
\usepackage[margin=3.5cm]{geometry}

\usepackage{fancyhdr}
\usepackage{hyperref}

\newcommand{\Z}{{\mathbb Z}}
\newcommand{\N}{{\mathbb N}}

\newtheorem{theorem}{Theorem}
\newtheorem{lemma}[theorem]{Lemma}

\newtheorem{corollary}[theorem]{Corollary}

\newtheorem{claim}[theorem]{Claim}

\newtheorem*{definition*}{Definition}

\numberwithin{equation}{section}
\numberwithin{theorem}{section}


\title[Polynomial Schur's theorem]%
  {Polynomial Schur's theorem}

\author{Hong Liu}
\email{hongliu@ibs.re.kr}
\address{Extremal Combinatorics and Probability Group (ECOPRO), Institute for Basic Science (IBS), Daejeon, South Korea}

\author{P\'eter P\'al Pach}
\email{ppp@cs.bme.hu}
\address{MTA-BME Lend\"ulet Arithmetic Combinatorics Research Group, Department of Computer Science and Information Theory, Budapest
  University of Technology and Economics, 1117 Budapest, Magyar tud\'osok
  k\"or\'utja 2., Hungary and Department of Computer Science and DIMAP, University of Warwick, Coventry CV4 7AL, UK}

\author{Csaba S\'andor}
\email{csandor@math.bme.hu}
\address{Institute of Mathematics, Budapest University of Technology and Economics, H-1529 B.O.
Box, Hungary and MTA-BME Lend\"ulet Arithmetic Combinatorics Research Group}

\thanks{H.L.\ was supported by the Institute for Basic Science (IBS-R029-C4), the Leverhulme Trust Early Career Fellowship~ECF-2016-523 and the UK Research and Innovation Future Leaders Fellowship MR/S016325/1.\\
P. P. P. was partially supported by the National Research, Development and Innovation Office NKFIH
(Grant Nr.~PD115978) and the J\'anos Bolyai Research
Scholarship of the Hungarian Academy of Sciences. The author has also received funding from the European Research Council (ERC) under the European Union’s Horizon 2020 research and innovation programme (grant agreement No 648509). This publication reflects only its author's view; the European Research Council Executive Agency is not responsible for any use that may be made of the information it contains.\\ 
C.S. was supported by the OTKA
Grant No. K109789 and the J\'anos Bolyai Research Scholarship of the Hungarian Academy of Sciences.}

\begin{document}

\begin{abstract}
We resolve the Ramsey problem for $\{x,y,z:x+y=p(z)\}$ for all polynomials $p$ over $\Z$. In particular, we characterise all polynomials that are $2$-Ramsey, that is, those $p(z)$ such that any $2$-colouring of $\N$ contains infinitely many monochromatic solutions for
$x+y=p(z)$. For polynomials that are not $2$-Ramsey, we characterise all $2$-colourings of $\N$ that are not $2$-Ramsey, revealing that certain divisibility barrier is the only obstruction to $2$-Ramseyness for $x+y=p(z)$. 
\end{abstract}

\maketitle

\section{Introduction}
The study of Ramsey theory searches for monochromatic patterns in finite colourings of $\N$. A pattern is \emph{$k$-Ramsey}, $k\in\N$, if it appears \emph{infinitely} often in any $k$-colouring of $\N$; and \emph{Ramsey} if this holds for every $k\in\N$. Ramsey theory has a long history dating back to the famous theorem of Schur~\cite{Sch} in 1916, which states that the equation $x+y=z$ is Ramsey, that is, any finite colouring of $\N$ contains infinitely many monochromatic solutions to $x+y=z$. Another classical example is van der Waerden’s theorem~\cite{vdW} stating that $\{x, x+y,\ldots, x+(\ell-1)y\}$ is Ramsey for any $\ell\in\N$. Rado~\cite{Rad33} later in his seminal work resolved the Ramsey problem for all \emph{linear} equations, characterising all those that are Ramsey. Since then, many extensions have been studied, see e.g.~the far-reaching polynomial extension of van der Waerden’s theorem by Bergelson and Leibman~\cite{BL96}.

In this paper, we study the polynomial extension of Schur's theorem. Somewhat surprisingly, only a special case of this natural problem has been solved. Csikv\'ari, Gyarmati and S\'ark\"ozy~\cite{CGS} showed that $x+y=z^2$ is \emph{not} $16$-Ramsey, that is, they constructed a $16$-colouring of $\N$ with no monochromatic solution for $x+y=z^2$ other than the trivial solution $x=y=z=2$. Later, Green and Lindqvist~\cite{GL} completely resolved this case using Fourier-analytic arguments, giving the satisfying answer that any $2$-colouring of $\N$ contains \emph{infinitely} many monochromatic solutions, while $3$ colours suffice to avoid non-trivial monochromatic solutions. In other words, $x+y=z^2$ is $2$-Ramsey, but not $3$-Ramsey. In fact, the $3$-colouring in~\cite{GL} can be easily adapted to show that 
\begin{itemize}
	\item[] \emph{$x+y=p(z)$ is not $3$-Ramsey for any $p(z)\in\Z[z]$ with $\deg(p)\ge 2$.}
\end{itemize} 
The result in~\cite{GL} also implies that there are at least $\log\log N$ monochromatic solutions in $[N]:=\{1,\ldots,N\}$ for any sufficiently large $N$. On the other hand, there is a greedy 2-colouring with at most $N^{1/2}$ monochromatic solutions. Recently, the second author~\cite{Pach} gave a shorter combinatorial proof for the 2-Ramseyness of $x+y=z^2$.




What can we say about a generic polynomial? A priori, it is not even clear, for example, whether other degree-2 polynomials are also $2$-Ramsey. After a moment (or two!) of thoughts, it is not hard to realise that this is certainly \emph{not} the case for \emph{all} quadratic polynomials due to a parity obstruction, as witnessed by the following example. 

\medskip 

\noindent\emph{Example 1.} Consider 
$$p(z)=2z^2+1.$$ 
Note that $p(z)$ takes only odd values, to have $x+y=2z^2+1$, it must be that $x$ and $y$ have different parities. As a result, one can 2-colour numbers in $\N$ according to their parities to avoid any monochromatic triple $\{x,y,z:x+y=2z^2+1\}$.

\subsection{Main results}
We completely resolve the Ramsey problem for $$\{x,y,z:x+y=p(z)\}$$ 
for all polynomials $p$ over $\Z$, thereby establishing a polynomial extension of Schur's theorem. In particular, we characterise all polynomials that are $2$-Ramsey. 

In fact, our results are stronger. For polynomials that are $2$-Ramsey, we prove a quantitative result, giving a lower bound on the number of monochromatic solutions. 

\begin{theorem}\label{thm-numberofsol}
	Let $p(z)=a_dz^d+\dots+a_1z+a_0\in\mathbb{Z}[z]$ with $d\ge 1$ and $a_d>0$ such that $2\mid p(1)p(2)$.
	Let $\phi$ be a 2-colouring of $[n]$. Then the number of monochromatic solutions $\{x,y,z\}\in [n]^{(3)}$ to $x+y=p(z)$ is at least $n^{{2}/{d^3}-o(1)}$. Moreover, there is a 2-colouring for which the number of monochromatic solutions is only $O(n^{2/d^2})$.
\end{theorem}

Note that the condition $a_d>0$ is necessary as otherwise $p(z)$ would eventually take only negative values. The assumption $2\mid p(1)p(2)$ is also needed, since otherwise $p(z)\equiv p(1)p(2)\equiv 1\pmod{2}$ and one can 2-colour $\N$ by parities as in Example 1 to avoid monochromatic solutions. 

It remains an interesting question to close the gap between the lower bound $n^{{2}/{d^3}-o(1)}$ and the upper bound  $O(n^{2/d^2})$ on the number of monochromatic solutions.

On the other hand, for polynomials that are not $2$-Ramsey, we characterise all $2$-colourings of $\N$ that are not $2$-Ramsey, showing that all such \emph{bad} $2$-colourings have to be \emph{balanced} and \emph{periodic}. Moreover the sumset of each colour class must have a rigid structure. It further reveals that a \emph{divisibility} barrier, generalising the aforementioned parity obstruction, is the \emph{only} obstruction to $2$-Ramseyness for $x+y=p(z)$. 

\begin{theorem}\label{thm-char}
	Let $p(z)=a_dz^d+\dots+a_1z +a_0\in\mathbb{Z}[z]$, with $d\ge 1$ and $a_d>0$. Let  $\phi:\mathbb{N}\to\{-1,1\}$ be a 2-colouring such that $x+y=p(z)$ does not have infinitely many monochromatic solutions. Then there exist an even positive integer $m$ and a partition of $\mathbb{Z}_m$ into two classes $A$ and $B$, each of size $m/2$, such that 
	$$\phi(x)=-1\quad \mbox{if and only if}\quad  x\in A\pmod{m}.$$ 
	Furthermore, there exists an odd $\alpha\in\Z_m$ such that 
	$$A+A=B+B=\mathbb{Z}_m\setminus \{\alpha\},$$
	and for any $z\in\N$, we have
	$$p(z)\equiv \alpha \pmod{m}.$$ 
\end{theorem}

Note that if $\phi$ and $p$ satisfies the above conditions, then $p(z)\equiv \alpha \pmod{m}$ for every $z$, however, whenever $x$ and $y$ have the same colour $x+y\not\equiv \alpha \pmod{m}$. Thus there is no monochromatic solution, even the trivial ones. In other words, if $x+y=p(z)$ has a trivial solution, such as $x=y=z$ for $x+y=z^2$, then the polynimial $p$ is necessarily $2$-Ramsey. We thus have the following corollary.

\begin{corollary}\label{cor:icecream}
	Let $p(z)=a_dz^d+\dots+a_1z +a_0\in\mathbb{Z}[z]$ with $d\ge 1$ and $a_d>0$ and $\phi$ be a 2-colouring of $\mathbb{N}$. Either there is no monochromatic solution for $x+y=p(z)$, or there are infinitely many monochromatic solutions.
\end{corollary}

A special case of the periodic colouring is the one induced by parity and a polynomial, e.g. the one in Example 1, for which $p(1)p(2)$ is always odd. Below is another example illustrating the divisibility barrier to 2-Ramseyness for $x+y=p(z)$.

\medskip 

\noindent\emph{Example 2.} Consider 
$$p(z)=z^3+3z^2+2z+3=z(z+1)(z+2)+3.$$ 
Note that for every $z\in\N$,
$$p(z)\equiv 3\pmod{6}.$$
Colour all numbers that are $2,3,5$ modulo $6$ with one colour, and the rest, $0,1,4$ modulo $6$, with the other colour. One can easily check that any number that is $3\pmod{6}$ cannot be written as a sum of two numbers of the same colour.

\medskip

Theorem~\ref{thm-char} also has the following corollary, characterising all polynomials that are $2$-Ramsey.

\begin{corollary}\label{cor-iff-infinite}
	Let $p(z)=a_dz^d+\dots+a_1z +a_0\in\mathbb{Z}[z]$ with $d\ge 1$ and $a_d>0$. Then every 2-colouring of $\mathbb{N}$ has infinitely many monochromatic solutions to $x+y=p(z)$ if and only if $p(1)\cdot p(2)$ is even.
\end{corollary}

\begin{proof}
	If $p(1)p(2)$ is odd, then $p(z)$ is always odd, hence for the colouring induced by parity there is no monochromatic solution at all. In fact, there is no solution to $x+y=p(z)$ with $x,y$ having the same colour.
	
	For the only if statement, by Theorem~\ref{thm-char}, if a 2-colouring does not have infinitely many monochromatic solutions, then for some even $m$ and odd $\alpha\in \Z_m$, we have for every $z$, $p(z)\equiv \alpha\pmod{m}$, implying that $p(z)\equiv 1\pmod{2}$. Hence $p(1)p(2)$ is odd.  
\end{proof}

Our proof is robust enough to prove a quantitative version of $2$-Ramseyness in which monochromatic solutions are guaranteed in the interval $[N, N^{d^3+o(1)}]$, see Theorem~\ref{thm-interval} in the concluding remarks. In fact, our method is also applicable to a large family of sub-exponential functions, see Theorem~\ref{thm-functions}.


\subsection{Remark}
To attack the Ramsey problem $x+y=p(z)$, our starting point is a clever idea (see Lemma~\ref{lem-monotone}) of~\cite{Pach} for the 2-Ramseyness of $x+y=z^2$, showing the link between certain  monotonicity phenomenon and existence of monochromatic solutions in a $2$-colouring $\phi$. Our proof for the general problem then differs after this point. Substantial new ideas are introduced to overcome various difficulties we encounter for a generic polynomial. 

Indeed, for example, considering some $k\in\N$ with $\phi(k)\ne \phi(k+1)$, the argument in~\cite{Pach} gave for the case $x+y=z^2$ that each element (up to $k^2$) with mod $2k+1$ residue in the interval $I=(o(k),(1-o(1))k)$  has colour $\phi(k)$. From here, it is not difficult to see that $k^2$ can be expressed as a sum of two numbers from residue classes contained in $I$, which yields a monochromatic solution with $z=k$. This is quite natural taking into account the fact that $|\mathbb{Z}_{2k+1}\setminus (I+I)| =o(k)$, which implies that almost all residues are contained in $I+I$. In contrast, for $x+y=p(z)$ with an arbitrary polynomial $p(z)$, the analogous conclusion only holds for the same interval $I$ modulo $m(k)=p(k+1)-p(k)$. As we choose ``larger'' polynomials (e.g. polynomials with higher degrees), this conclusion gets much weaker. For instance, even by taking a degree-2 polynomial with larger leading coefficient, say $p(z)=100z^2$, the interval $I+I$ covers less than $1\%$ of $\mathbb{Z}_{m(k)}=\mathbb{Z}_{200k+100}$. For a polynomial $p(z)$ with degree $d$, we have $|I|=(1-o(1))k$ and $m(k)=\Theta(k^{d-1})$, thus for large degree $d$ we have information on the colour of a very sparse subset. Further difficulties will be discussed in the next section, where an overview of our methods is presented.

\subsection{Other related work}

It is worth noting that the Ramsey problem for $x^{\alpha}+y^{\beta}=z^{\gamma}$ in $\Z/p\Z$ has been studied by Lindqvist~\cite{Lin}. If one puts no restriction on $z$ and looks for monochromatic pair $\{x,y\}$ with $x+y$ being a perfect square, then Khalfallah and Szemer\'edi~\cite{KS} showed that this is Ramsey in $\N$. Yet another similar looking pattern that behaves very differently is to consider $x-y$ instead. Bergelson~\cite{Ber}, improving upon results of Furstenberg~\cite{Fur77} and S\'ark\"ozy~\cite{Sar78}, proved that $\{x,y,z:x-y=z^2\}$ is Ramsey.

Ramsey theory has witnessed exciting development recently. We refer the readers to the papers of Green and Sanders~\cite{GS16} and of Moreira~\cite{Mor17} for the problem involving sum and product of $x$ and $y$, and to the papers of Di Nasso and Luperi Baglini~\cite{NB18}, and of Chow, Lindqvist and Prendiville~\cite{CLP18+} for  generalisations of Rado’s criterion to non-linear polynomials.

\medskip

\noindent\textbf{Organisation.} The rest of the paper is organised as follows. We first give an overview of our methods in Section~\ref{sec-overview}. We then present the proofs of the two main results, Theorems~\ref{thm-char} and~\ref{thm-numberofsol} in Sections~\ref{sec-char} and~\ref{sec-numberofsol} respectively. Some concluding remarks are given in Section~\ref{sec-remark}.

\section{Overview of the methods}\label{sec-overview}
We present in this section the proof sketch for our main results:
characterising all pairs of polynomials $p$ and 2-colourings $\phi$ such that $x+y=p(z)$ does not have any (or equivalently\footnote{For this equivalence, see Corollary~\ref{cor:icecream} and the paragraph before it.}, does not have infinitely many) $\phi$-monochromatic solutions (Theorem~\ref{thm-char}); and a lower bound on the number of monochromatic solutions in $[n]$  (Theorem~\ref{thm-numberofsol}). Note that both results imply the 2-Ramseyness of the equation $x+y=p(z)$ when $p(1)p(2)$ is even. 

We start with sketching the proof of Theorem~\ref{thm-char}.
Trivially, if there is a ``very long'' monochromatic interval, then many monochromatic solutions can be found within it. Thus, we may assume that there will be infinitely many places where the colour switches. With the help of a simple, but crucial observation we can see that whenever a ``sufficiently long'' block of numbers of one colour is followed by a sufficiently long block of numbers coloured with the other colour, many monochromatic solutions can be found. This allows us to assume that the colour switches ``frequently'' after some threshold.

When considering a switch $k$, i.e.~$\phi(k)\ne \phi(k+1)$, we define a subset $A=A_k\subseteq\mathbb{Z}_{m(k)}$ (where $m=m(k):=p(k+1)-p(k)$) containing at least half of the elements of $\mathbb{Z}_{m}$. The set $A$ satisfies that whenever $z\in\N$ is such that (i) $p(z)$ lies in the sumset $A+A\pmod{m}$, and (ii) $z$ has the opposite colour of $k$, then we are able to find a monochromatic solution. To drop the restriction (ii) on the colour of $z$, we shall use that the colour switches frequently, according to the above discussion. If $k_1$ and $k_2$ are two consecutive switches, then clearly $\phi(k_1)=-\phi(k_2)$ and either $k_1$ or $k_2$ would have the opposite colour of $z$.

However, we still need to guarantee (i) that $p(z)\in A+A\pmod{m}$. As $A\subseteq \Z_{m}$ contains at least $m/2$ elements, by the pigeon-hole principle $A+A=\Z_{m}$ holds if $|A|\ne m/2$, and then $p(z)\in A+A$ is automatically satisfied. If $|A|=m/2$, then the sumset $A+A$ might not contain all elements of $\Z_{m}$. These cases are described in Lemma~\ref{lem-sumset-char} (a stability version of Cauchy-Davenport theorem), and indeed the union of the residue classes outside of the sumset $A+A$ form a residue class $\alpha$ modulo $m'$ for some even $m'|m$.

Now, if we obtain the same  $\alpha, m'$ infinitely often, then this forces the periodic structure of the colouring and $p(z)\equiv \alpha \pmod {m'}$ for all $z$. Otherwise we would get a sequence $m'\to \infty$. However, for a fixed polynomial $p$ it is not possible to have $p(z)\equiv \alpha \pmod {m'}$ for all $z$ if $m'$ is sufficiently large.  More precisely, with the help of Szemer\'edi's theorem on arithmetic progressions, we prove this in Lemma~\ref{lem-cong} for a pair of moduli $m'_1,m'_2$, as to drop the condition on $\phi(z)$ we work with pairs of switches. 

To prove Theorem~\ref{thm-numberofsol} we use some of the ideas arising so far, e.g. it still holds that the colour switches frequently. However, we need to overcome additional difficulties. One such obstacle is that $|A|\geq m/2$ can not be assumed (for $A\subseteq \Z_{m})$, since finding a single monochromatic solution is not enough any more. Our task now is to find ``a lot of'' monochromatic solutions. Thus instead of $|A|\geq m/2$, we are only able to assume $|A|\geq (1/2-o(1))m$. With the help of Kneser's theorem we can see that this weaker condition is still enough to show that either the sumset $A+A$  contains at least $1-o(1)$ portion of $\Z_{m}$ (Claim~\ref{cl-goodclasses}) or we can find a large number of monochromatic solutions.

Here we also need to find ``many'' $z$ with $p(z)\in A+A\pmod{m}$, that is, we shall see that it is not possible that the majority of the elements of a long interval $I$ are mapped by the polynomial $p$ to a small set of residue classes modulo $m$, provided that $m$ is sufficiently large. This is shown in Lemma~\ref{henseltype} which is partially parallel in spirit with Hensel's lemma. 
Alternatively, we can think of this lemma as claiming that a fixed polynomial can not have ``too many'' (more than $m^{o(1)}$) roots within a certain dense subset (which can be explicitly given) of $\mathbb{Z}_m$. We believe this lemma about modular arithmetic is interesting on its own right. 

\section{Characterisation of $2$-Ramsey polynomials}~\label{sec-char}
In this section, we prove Theorem~\ref{thm-char}. We start with Lemma~\ref{lem-monotone}, showing that certain monotonicity must appear when there is no monochromatic solutions. Then we prove Lemma~\ref{lem-sumset-char}, a stability version of Cauchy-Davenport theorem, describing the structure of a maximum-size subset of $\Z_m$ whose sumset is not the whole $\Z_m$. Next, we present the final ingredient, Lemma~\ref{lem-cong}, stating that the image of long intervals under the polynomial map is ``large'' in a sense that it can avoid certain residue classes. Lemmas~\ref{lem-sumset-char} and~\ref{lem-cong} will be used to show that 2-colourings without infinitely many monochromatic solutions must be periodic. Throughout the proof, the following function will play a central role. Let 
$$m(k):=p(k+1)-p(k).$$ 
Note that $m(k)$ is a degree-$(d-1)$ polynomial with positive leading coefficient $a_d{d\choose 1}>0$.

\subsection{Monotonicity}\label{subsec-monotone}
Given a 2-colouring $\phi:\mathbb{N}\to \{-1,1\}$, we say that an integer $k\in\N$ is a \emph{switch} (for $\phi$), if $\phi$ changes colour at $k$, i.e.~$\phi(k)\ne \phi(k+1)$. 

Consider now a sufficiently large switch $k$ (in turns of coefficients of $p(z)$) with $\phi(k)=1$ and $\phi(k+1)=-1$, so that $m(k)$ is positive. For each $j\in \Z_{m(k)}$, denote by $L_j$  the integer such that 
$$p(k)\le j+(L_j+1)m(k)< p(k+1).$$
Note that by definition, $j+L_jm(k)<p(k)$. We define a set $H_j$ to be the residue class $j$ modulo $m(k)$ up to $p(k)$, that is, 
$$H_j:=\{j,j+m(k),\dots,j+L_jm(k)\}.$$
A residue class $j$ modulo $m(k)$ is \emph{monotone} if $\phi$ is non-decreasing\footnote{Non-increasing if the switch $k$ is such that $\phi(k)=-1$ and $\phi(k+1)=1$.} on $H_j$:
\begin{equation*}
\phi(j)\leq \phi(j+m(k))\leq \phi(j+2m(k))\leq\dots\leq \phi(j+L_jm(k)).
\end{equation*}
We call a switch $k$ \emph{monotone} if every residue class $j\in\Z_{m(k)}$ is monotone. 
\begin{lemma}\label{lem-monotone}
	Let $\phi:\mathbb{N}\to \{-1,1\}$ be a 2-colouring and $k$ be a switch. If $k$ is not monotone, then $\phi$ contains a monochromatic solution to $x+y=p(z)$.
	
	Furthermore, for the same switch $k$, violating the monotonicity at different residue classes yields distinct monochromatic solutions; and for any switch $k'>k+1$, violations of the monotonicities of $k$ and $k'$ correspond to distinct monochromatic solutions.
\end{lemma}
\begin{proof}
	Assume without loss of generality that $\phi(k)=1$. Suppose some residue class $j\in \Z_{m(k)}$ is not monotone, say 
	$$\phi(j+\ell m(k))>\phi(j+(\ell+1) m(k))$$ 
	for some $\ell< L_j$, then 
	\begin{eqnarray*}
		&&\phi(j+\ell m(k))+\phi(p(k)-(j+\ell m(k)))\\
		&>&\phi(j+(\ell+1) m(k))+\phi(p(k)-(j+\ell m(k)))\\
		&=&\phi(j+(\ell+1) m(k))+\phi(p(k+1)-(j+(\ell+1) m(k))).
	\end{eqnarray*}
	Note that for any $a,b\in\N$, $\phi(a)+\phi(b)$ takes value in $\{-2,0,2\}$. Since at least one side of the above inequality is non-zero, we see that either both $j+\ell m(k)$ and $p(k)-(j+\ell m(k))$ are of colour 1; or both $j+(\ell+1) m(k)$ and $p(k+1)-(j+(\ell+1) m(k))$ are of colour $-1$. We then get a monochromatic solution in either case: as ordered triples
	$$\{x,y,z\}=\{j+\ell m(k), ~p(k)-(j+\ell m(k)), ~k\}$$ 
	or 
	$$\{x,y,z\}=\{j+(\ell+1) m(k), ~p(k+1)-(j+(\ell+1) m(k)), ~k+1\}.$$
	
	As $x$ determines the residue class $j$ and $z\in\{k,k+1\}$ in all the above solutions, the furthermore part is clear.
\end{proof}
\subsection{Periodicity}\label{subsec-period}
To describe the structure of a $2$-colouring without infinitely many monochromatic solutions to $x+y=p(z)$, we need the following lemma, which can be regarded as a stability version of Cauchy-Davenport theorem. It states that if the sumset of an $m/2$-set $A\subseteq \Z_m$ does \emph{not} cover the entire $\Z_m$, then its sumset must have certain periodic structure. In particular, $A+A$ contains exactly those residue classes mod $m$ that are \emph{not} contained in a certain residue class mod $m'$ for some even divisor $m'$ of $m$.

\begin{lemma}\label{lem-sumset-char}
Let $2\mid m$, $A\subseteq \mathbb{Z}_m$ be of size $m/2$ and $B=\Z_m\setminus A$. If $A+A\ne \mathbb{Z}_m$, then there exist $2\mid m'\mid m$ and an odd $\alpha\in \Z_{m'}$ such that the followings hold. Let $\varphi$ be the canonical homomorphism from $\Z_m$ to $\Z_{m'}$ and $A':=\varphi(A)$, $B'=\varphi(B)$. Then
$$A'+A'=B'+B'=\Z_{m'}\setminus\{\alpha\},$$
and 
$$A+A=B+B=\{r\in \Z_m: r\not\equiv \alpha \pmod{m'}\}.$$
Furthermore, for any $a\in A'$ and $b\in B'$, 
$$\alpha\in (b+A')\cap (a+B').$$
\end{lemma}

\begin{proof}
Let $X$ be the complement of $A+A$, that is,
$$X:=\Z_m\setminus (A+A)=\{x\in \mathbb{Z}_m: x\notin A+A\}\ne \emptyset,$$
and $H$ be the subgroup generated by $X-X$, i.e.
$$H:= \langle X-X\rangle \leq \mathbb{Z}_m.$$ 

We claim that $H$ is contained in the stabiliser of $A$ and $A+A$, i.e.
$$H+A=A, \quad \text{ and }\quad H+A+A=A+A.$$ 
To see this, observe first that $A\cap (x-A)=\emptyset$ for any $x\in X$, as otherwise $a=x-a'$ with $a,a'\in A$ would yield $x=a+a'\in A+A$, contradicting the definition of $X$. Consequently, as $|A|=|x-A|=m/2$, we have $x-A=\mathbb{Z}_m\setminus A=B$ and also that $x-B=A$. Fix now arbitrary $x_1,x_2\in X$. The above observation shows $x_1-A=x_2-A$, whence $(x_2-x_1)+A=A$. Therefore, $H+A=A$ and $(H+A)+A=A+A$, as claimed.

Consider now the quotient $K:=\mathbb{Z}_m/H$. Let $\varphi: \mathbb{Z}_m \to \mathbb{Z}_m/H=K$ be the canonical homomorphism. We shall show that $\varphi(X)$ is a singleton $\alpha$, which together with $m':=|K|$, satisfies the desired property.

Let $A':=\varphi(A)\subseteq K$, then $\varphi(A+A)=A'+A'$. Note that in fact $A=\varphi^{-1}(A')$ is a union of $H$-cosets, since if $g\in A$, then $H+g\subseteq H+A=A$. Similarly $A+A=\phi^{-1}(A'+A')$ is also a union of $H$-cosets. Then, as $|A|=m/2$, we have $|A'|=|K|/2=m'/2$.

Let 
$$X':=K\setminus (A'+A')=\{x\in K: x\notin A'+A'\}\ne\emptyset.$$ 
It follows from $\varphi(A+A)=A'+A'$ and $A+A=\varphi^{-1}(A'+A')$ that $X'=\varphi(X)$ and $X=\varphi^{-1}(X')$. By definition, $X\subseteq x+H$ for any $x\in X$, and so 
$$X'=\varphi(X)\subseteq \varphi(x+H)=\varphi(x)=:\alpha.$$
Thus, $X'=\varphi(X)=\{\alpha\}$, $A'+A'=\Z_{m'}\setminus \{\alpha\}$, and 
$$A+A=\Z_m\setminus X=\Z_m\setminus \varphi^{-1}(\alpha).$$
Recall that $x-B=A$ for all $x\in X$, reversing the roles of $A$ and $B$, the conclusion also holds for $B'+B'$ and $B+B$. 

Since $\alpha\not\in B'+B'$, for any $b\in B'$, we see that $\alpha-b\in A'$, that is, $\alpha\in b+A'$. Similarly $\alpha\in a+B'$ for any $a\in A'$. By Lagrange's theorem, $m'$, the order of $K$, divides $m$. Moreover, $K$ is cyclic and $|A'|=m'/2$ implies that $m'$ is even. 

We are left to show that $\alpha$ must be odd. If $\alpha=2\alpha'$ is even, then $\alpha',\alpha'+m'/2\notin A'$. Moreover, from each of the $m'/2-1$ pairs $\{x,\alpha-x\}$ with $x\notin \{\alpha',\alpha'+m'/2\}$, the set $A'$ can contain at most one element, contradicting $|A'|=m'/2$.
\end{proof}

To illustrate Lemma~\ref{lem-sumset-char}, we rephrase the two examples in the introduction.

\medskip

\noindent\emph{Example 3:} Let $2|m$. Consider $A=2\Z_m$, $B=2\Z_m+1$, then $m'=2$ and $\alpha=1$, and $A+A=B+B=\{r\in\Z_m: r\not\equiv 1  \pmod 2\}$.

\medskip

\noindent\emph{Example 4:} Let $6|m$. Consider $A=6\Z_m+\{2,3,5\}$, $B=6\Z_m+\{0,1,4\}$, then $m'=6$ and $\alpha=3$, and $A+A=B+B=\{r\in\Z_m: r\not\equiv 3  \pmod 6\}$.

\subsection{Images of long intervals are ``large''}\label{subsec-cong}
In search of monochromatic solutions, we often consider some ``nice'' residue classes. The next lemma allows us to avoid the ``bad'' classes. It states that for any polynomial $p\in\Z[z]$, the image of any interval $I$, $p(I)$, can avoid any residue classes as long as $I$ is sufficiently long. For its proof, we will use the following result from the theory of finite differences for polynomials. Let $p(x)\in\Z[x]$ be a polynomial of degree at most $n$, then for any $m,\ell \in \N$,
\begin{equation}\label{eq-binomial}
\sum_{i=0}^{n}(-1)^{i}{n\choose i}p(m+(n-i)\ell)=\ell^nn!\cdot a_n,
\end{equation}
where $a_n$ is the coefficient of degree $n$ in $p(x)$. (An elegant proof of this identity is given in \cite{Pohoata}.)

\begin{lemma}\label{lem-cong}
Let $p(z)=a_dz^d+\dots+a_1z +a_0\in\mathbb{Z}[z]$, with $d\ge 1$ and $a_d>0$. There exists $M,N$ such that for every $m_1,m_2>M$, every $\alpha_1,\alpha_2$ and every interval $I$ of length $N$ there exists $z\in I$ such that for each $i\in\{1,2\}$,
$$p(z)\not\equiv \alpha_i \pmod{m_i}.$$
\end{lemma}

\begin{proof}
By Szemer\'edi's theorem~\cite{Sze75}, there exists some $N$ such that if an interval $I$ has length $N$ and $X\subseteq I$ has size $|X|\geq \lfloor |I|/(d+1) \rfloor$, then $X$ contains an arithmetic progression of length $d+1$. Let $M=d!a_dN^d+1$.

First we show that in any interval of length $d+1$, there is a $z$ such that $p(z)\not\equiv \alpha_1 \pmod{m_1}$. Suppose to the contrary that $p(z)-\alpha_1\equiv 0 \pmod{m_1}$ for $d+1$ consecutive integers, say $a,a+1,\dots,a+d$. Then $m_1$ divides 
$$\sum\limits_{i=0}^d (-1)^i \binom{d}{i}(p(a+d-i)-\alpha_1)=d!\cdot a_d$$ 
due to~\eqref{eq-binomial}. This contradicts $m_1>M>d!\cdot a_d$.

Now, take an interval $I$ of length $N$. Let $X=\{z\in I: p(z)\not\equiv \alpha_1 \pmod{m_1}\}$. We have $|X|\ge \lfloor |I|/(d+1)\rfloor$, so $X$ contains an arithmetic progression of length $d+1$, say $b,b+c,\dots,b+dc$. As $|I|=N$, clearly $c<N$. Suppose to the contrary that $p(z)-\alpha_2\equiv 0 \pmod{m_2}$ for every $z\in \{b,b+c,\dots,b+dc\}$. Then $m_2$ divides 
$$\sum\limits_{i=0}^d (-1)^i \binom{d}{i}(p(b+(d-i)c)-\alpha_2)=c^dd!\cdot a_d$$ 
due to~\eqref{eq-binomial}. This contradicts $m_2>M>d!a_dN^{d}$.
\end{proof}

\subsection{Proof of Theorem~\ref{thm-char}}\label{subsec-proof-char}
Let $\phi:\mathbb{N}\to \{-1,1\}$ be a 2-colouring such that $x+y=p(z)$ does not have infinitely many monochromatic solutions. 

First note that if there is no switch in a sufficiently long interval, say $[n,Kn^{d}]$ with $Kn^{d}>p(n)$, then we can find a monochromatic solution in this interval. Therefore, we may assume that there are infinitely many switches. Let $\{k_i\}$ be an infinite sequence of switches with alternating colours and $10N\le k_1\le k_2\le\ldots$, where $N$ is the constant from Lemma~\ref{lem-cong}.

By Lemma~\ref{lem-monotone}, distinct non-monotone switches yield distinct monochromatic solutions. We may therefore assume that there are only finitely many non-monotone switches. By passing to a subsequence, we may assume that each switch in $\{k_i\}$ is monotone.

Fix now a switch $k\in\{k_i\}$ with $\phi(k)=1$ and $\phi(k+1)=-1$. Since $k$ is monotone, for any $j\in \Z_{m(k)}$, the restriction $\phi\big|_{H_j}$ of $\phi$ on $H_j$ consists of two constant intervals, i.e.
\begin{equation}\label{eq-pattern}
	\phi\big|_{H_j}=\{-1,-1,\dots,-1,1,1,\dots,1\}, ~~~\mbox{for all } j\in\Z_{m(k)},\footnote{Analogously $\phi\big|_{H_j}=\{1,1,\dots,1,-1,-1,\dots,-1\}$ if $\phi(k)=-1$.}
\end{equation}
Note that the colour $1$ or the colour $-1$ interval might be empty.

We define a function $\beta(\cdot)$ indicating when the colour changes in $H_j$. For $j\in \Z_{m(k)}$, let $\beta(j)$ be the smallest element of $H_j$ with colour $\phi(k)=1$. That is, $\beta(j)=j+\ell m(k)$, if $\phi(j)=\dots=\phi(j+(\ell -1)m(k))=-1$ and $\phi(j+\ell m(k))=\dots=\phi(j+L_jm(k))=1$. If $\phi\big|_{H_j}$ is monochromatic in colour $-\phi(k)=-1$, then set $\beta(j)=\infty$.
We further extend $\beta(\cdot)$ and $H_j$ periodically to the set of all natural numbers: $\forall j'\in\N$, let $\beta(j'):=\beta(j)$ where $j\in\Z_{m(k)}$ and $j'\equiv j\pmod{m(k)}$, and set $H_{j'}:=H_j$. 

We introduce a set $A_k$ consisting of all residue classes mod $m(k)$ with large $\beta$-value, i.e.~with long initial segments of colour $-\phi(k)=-1$:
$$A_k:=\{j\in \Z_{m(k)}: \beta(j)\geq p(k)/3\}.$$
We next show that either $A_k+A_k$ covers $\mathbb{Z}_{m(k)}$, in which case we call the switch $k$ \emph{good}; or $A_k$ must have a periodic structure as described in Lemma~\ref{lem-sumset-char}, and we call such $k$ \emph{bad}. 

\begin{claim}\label{cl-Ak}
	We may assume that for each $k\in\{k_i\}$, either $A_k+A_k=\Z_{m(k)}$, or $|A_k|=m(k)/2$.
\end{claim}
\begin{proof}
	Notice first that there are only finitely many switches $k$, for which there exists some $j\in \Z_{m(k)}$ with
	 $$\beta(j)+\beta(p(k)-j)\leq 2p(k)/3.$$ 
	Indeed, fix one such pair $k$ and $j$. For any $x\in H_j$ and $y\in H_{p(k)-j}$, we have $x+y\equiv p(k)\pmod{m(k)}$. By~\eqref{eq-pattern} and the definition of $\beta(\cdot)$, we see that every number $n$ with $n\equiv p(k)\pmod{m(k)}$ in 
	$$[\beta(j)+\beta(p(k)-j), 2p(k)-2m(k)]\supseteq [2p(k)/3, 2p(k)-2m(k)]$$ 
	can be written as the sum of two numbers of colour $\phi(k)=1$. We then obtain a monochromatic solution by choosing $x\in H_j$ and $y\in H_{p(k)-j}$ with $x+y=p(k)$ and setting $z=k$.
	
	We may now assume that for every switch $k$ in $\{k_i\}$, $\beta(j)+\beta(p(k)-j)> 2p(k)/3$ for every $j\in m(k)$, whence $|A_k|\geq m(k)/2$. If $A_k+A_k$ does not cover $\Z_{m(k)}$, say $x\in \Z_{m(k)}$ is not in $A_k+A_k$, then $A_k\cap (x-A_k)=\emptyset$, implying that $|A_k|=m(k)/2$ as desired.
\end{proof}

If the switch $k$ is bad, then by Lemma~\ref{lem-sumset-char}, there is an even $m'(k)\mid m(k)$, an $A_k'\subseteq \mathbb{Z}_{m'(k)}$ of size $m'(k)/2$ and an odd $\alpha_k\in\Z_{m'(k)}$ such that $A_k'+A_k'=\Z_{m'(k)}\setminus\{\alpha_k\}$, and 
$A_k+A_k$ covers all residue classes mod $m(k)$ except those in the mod $m'(k)$ residue class of $\alpha_k$. Therefore, by the definition of $A_k$, every integer 
$$n\in [2m(k), 2p(k)/3-2m(k)]$$ 
can be written as a sum of two numbers of colour $-\phi(k)=-1$ from two residue classes of $A_k$, except when $k$ is bad and $n\equiv \alpha_k \pmod{m'(k)}$.


We first deal with the case when there are finitely many bad switches $k_i$ with bounded $m'(k_i)$. By passing to a subsequence, while preserving that consecutive switches have opposite colours, we may assume that for every bad switch $k_i$, $m'(k_i)>M$, where $M$ is the constant from Lemma~\ref{lem-cong}. 

\medskip

\noindent\textbf{Case 1:} \emph{for every $k\in \{k_i\}$, if $k$ is bad, then $m'(k)>M$}.

\medskip

For each $i$, define
$$I_i:=\{z: p(z)\in [2m(k_i), 2p(k_i)/3-2m(k_i)], \mbox{ and } p(z)\not\equiv \alpha_{k_i}\pmod{m'(k_i)} ~\mbox{ if $k_i$ is bad}  \}.$$
By the discussion after Claim~\ref{cl-Ak} and the definition of $I_i$, we see that, for every $z\in I_i$, $p(z)$ can be written as a sum of two numbers of colour $-\phi(k_i)$. In other words, if $I_i$ is not monochromatic in colour $\phi(k_i)$, we will get a monochromatic solution. Since $\phi$ does not contain infinitely many monochromatic solutions, we may then assume that 
\begin{itemize}
	\item[$(\ast)$] \emph{for each $i$, $I_i$ is monochromatic in colour $\phi(k_i)$.}
\end{itemize}

By taking $k_1$ sufficiently large, we may also assume that the degree-$(d-1)$ polynomial $m(k)$ satisfies $m(k)=o(p(k))$, and that the degree-$d$ term in $p(k)$ dominates, i.e.~$p(k)-a_dk^d=o(a_dk^d)$. Thus, for some $c\approx (\frac{2}{3})^{1/d}\ge \frac{2}{3}$, the set $I_i$ contains the interval $[ck_i/4,ck_i]$, apart from the $p$-preimage of the residue class $\alpha_{k_i}\pmod{m'(k_i)}$, if $k_i$ is bad.

We next show that consecutive switches must be \emph{far apart}. Consider two consecutive switches  $k_i,k_{i+1}$. Suppose $k_i<k_{i+1}<2k_i$, then from the above discussion $I_{i}\cap I_{i+1}$ contains the interval $[ck_i/2,ck_i]$, apart from the $p$-preimage of the residue class $\alpha_{k_j}\pmod{m'(k_j)}$ if $k_j$ is bad for $j\in\{i,i+1\}$. As $k_i\ge 10N$, Lemma~\ref{lem-cong} guarantees a number $z\in [ck_i/2,ck_i]$ such that $p(z)\not\equiv \alpha_{k_j} \pmod{m'(k_j)}$ for any $j\in \{i,i+1\}$ such that $k_j$ is bad. In other words, $I_i\cap I_{i+1}$ is non-empty, and by~$(\ast)$, $z\in I_i\cap I_{i+1}$ must have the same colour as $k_i$ and also $k_{i+1}$.  However, $k_i$ and $k_{i+1}$ are two consecutive switches with opposite colours, $\phi(k_i)\ne \phi(k_{i+1})$, a contradiction.

Hence, $2k_i\leq k_{i+1}$ for any $i$. Take now three consecutive switches, say without loss of generality $k_1<k_2<k_3$ with  $\phi(k_2)=1$. So $(k_1,k_2]$ is of colour 1; while $(k_2,k_3]$ is of colour $-1$. Consider $z$ with 
$$p(z)\approx 1.6k_2,$$
and so $z\approx (\frac{1.6k_2}{a_d})^{1/d}$.

If $\phi(z)=-1=\phi(k_3)$, then there is an $x\in [p(z)-k_2-1]$ of colour $\phi(z)=-1$. Note that  
$$[p(z)-k_2-1]\supseteq [0.59k_2]$$ 
contains a long interval, so such an $x$ must exist. Then $y=p(z)-x\in (k_2,k_3]$ is also of colour $-1$, and $x$, $y=p(z)-x$, $z$ form a monochromatic solution.

If $\phi(z)=1=\phi(k_2)$, then 
$$\phi(z)=\phi\left(\left\lfloor \frac{p(z)}{2} \right\rfloor\right)=\phi\left(\left\lceil \frac{p(z)}{2} \right\rceil\right)\quad \mbox{as}\quad \left\lfloor \frac{p(z)}{2} \right\rfloor, \left\lceil \frac{p(z)}{2} \right\rceil\approx 0.8k_2\in (k_1,k_2]$$
 due to $k_2\ge 2k_1$. So $x=\lfloor \frac{p(z)}{2} \rfloor,y=\lceil \frac{p(z)}{2} \rceil,z$ form a monochromatic solution.

We thus obtain infinitely many monochromatic solutions from either pairs of switches that are close to each other or triples of switches that are pairwise far apart.

\medskip

\noindent\textbf{Case 2:}  \emph{there are  infinitely many bad $k_i$ with $m'(k_i)\le M$.} 

\medskip

By passing to a subsequence, we may assume in this case that for every $i\in\N$, 
$$\phi(k_i)=1, \quad m'(k_i)=:m', \quad A_{k_i}'=:A', \quad \alpha_{k_i}=:\alpha$$ 
for some $m'\le M$, $A'\subseteq \Z_{m'}$ and odd $\alpha\in\Z_{m'}$.

Since $p(k_i)\rightarrow \infty$ as $k_i\rightarrow\infty$, $\beta(a)=\infty$ for every $a\in A'$. In other words, for any $n\in\N$, if $n\in A' \pmod{ m'}$, then $\phi(n)=-1$.

On the other hand, for each $b\in B':=\Z_{m'}\setminus A'$, note that the colours in the residue class $n\equiv b \pmod{m'}$ change before $p(k_1)/3$ from $-1$ to $1$ and remains positive as $p(k_i)\rightarrow \infty$. In other words, for any integer $n>p(k_1)/3$, if $n\in B' \pmod{m'}$, then $\phi(n)=1$. To show that $\phi$ is periodic, we still need to show small values of $n\in B'\pmod{m'}$ also get colour $1$. For this, we need to first show that 
$$p(z)\equiv \alpha \pmod{m'},$$ 
for every $z\in\N$.

Suppose to the contrary that for some $z\in\N$, $p(z)\not\equiv \alpha\pmod{m'}$. Note that $p(z)$ is periodic mod $m'$ with period $m'$, that is, for any $\ell\in \N$, 
$$p(z+\ell\cdot m')\equiv p(z)\pmod{m'}.$$ 
Therefore, from the existence of a number $z$ with $p(z)$ not congruent to $\alpha$ mod $m'$, it follows that there are infinitely many $z$ with $p(z)$ not congruent to $\alpha$ mod $m'$. For each such $z$ with $p(z)>2p(k_1)/3$, we have 
$$p(z)\in A'+A'=B'+B' \pmod{m'}$$ 
due to Lemma~\ref{lem-sumset-char}. We then get a monochromatic solution by writing $p(z)$ as a sum of two numbers of colour $-1$ (or colour $1$ resp.) from residue classes in $A'\pmod{m'}$ (or in $B'$ resp.), and clearly distinct choices of $z$ yield distinct solutions. We then get infinitely many monochromatic solutions, a contradiction.

Suppose now that for some $n\in B'\pmod{m'}$, $\phi(n)=-1$. By Lemma~\ref{lem-sumset-char}, 
$$\alpha\in n+A'\pmod{m'}.$$ 
Then for each of the infinitely many choice of $z$ with $p(z)\equiv \alpha\pmod{m'}$, we get a monochromatic solution in colour $-1$ by setting 
$$x=n\quad \mbox{and } \quad y=p(z)-n\in A'\pmod{m'},$$ 
a contradiction.

This completes the proof of Theorem~\ref{thm-char}.



\section{Number of solutions}\label{sec-numberofsol}
In this section, we prove Theorem~\ref{thm-numberofsol}. In order to count the number of solutions, we need a quantitative strengthening of Lemma~\ref{lem-cong}, see Lemma~\ref{henseltype}. Throughout this section the constants hidden in the $o,O,\Omega$ notions depend only on the polynomial $p$.
\subsection{Roots of polynomials modulo $m$}
We will need a lemma which shows that a polynomial ``usually'' does not have ``too many'' roots modulo $m$. When $m=q$ is a prime, then this clearly holds, as the number of roots is at most the degree of the polynomial. For composite numbers the situation is not so clear. For instance, the polynomial $p(z)=z^2$ has $q$ roots modulo $m=q^2$ (here $q$ is a prime), that is, the number of roots might be as large as $\sqrt{m}$ for infinitely many values of $m$. However, one can easily check that for any $q\nmid c$ the polynomial $p(z)=z^2+c$ can have at most 2 roots modulo $q^2$. 

Our next lemma, extending Lemma~\ref{lem-cong} quantitatively, states that for any polynomial $p(z)$ and interval $I$, it is possible to choose a ``not too sparse'' subset $Z$ of $I$ such that $p(z)\equiv c\pmod{m}$ can have only $m^{o(1)}$ solutions with $z\in Z$. Note that this also implies that the image $p(I)$ (as a subset of $\mathbb{Z}_m$) must be ``large''. 

\begin{lemma}\label{henseltype}
Let $p(z)=a_dz^d+\dots+a_1z+a_0\in\mathbb{Z}[z]$ with $d\ge 1$ and $a_d>0$. There exists $m_0$ such that the following holds for any $m\ge m_0$ and any interval $I\subseteq \N$ of length $O(m)$. There exists a subset
$Z:=Z(I,m)\subseteq I$ of size at least 
$$|Z|=\Omega\left(\frac{|I|}{(\log\log m)^{d-1}}\right)$$
such that for every $c\in \Z_m$, the congruence $p(z)\equiv c\pmod {m}$ has at most 
$$e^{O(\frac{\log m}{\log\log m})}$$
solutions with $z\in Z$. Furthermore, for any $m^*|m$ with $m^*\ge m_0$ and $|I|=O(m^*)$, 
$$Z(I, m)\subseteq Z(I, m^*).$$
\end{lemma}

\begin{proof}
We may assume that $I$ is of length $e^{\Omega(\frac{\log m}{\log\log m})}$, otherwise we can simply take $Z=I$. We first construct a subset $Z:=Z(I,m)\subseteq I$ explicitly and then show that it has the desired properties.

\smallskip

\noindent\emph{Construction.} For every prime $q\mid m$, we will select a subset $S_q\subseteq \mathbb{Z}$ and the desired set $Z$ will be their common intersection with the interval $I$, i.e.~$Z:=I\cap \bigcap\limits_{q\mid m}S_q$.
To define $S_q$, two cases are distinguished based on the size of $q$: either $q\leq d-1$ or $d\leq q$.

\medskip

\noindent\emph{Case 1.} $q\leq d-1$. Choose an integer $\tau>0$ in such a way that the derivative $p'$ is not constant 0 modulo $q^\tau$, say 
$$p'(z_q)\not\equiv 0\pmod{q^\tau}.$$ 
Such a $z_q$ exists as $p'$ is a degree-$(d-1)$ polynomial with positive leading coefficient, $a_d{d\choose 1}>0$, and will eventually take positive value, say $p'(z_{q})>0$. Then any $\tau$ with $q^\tau>p'(z_{q})$ will do.

As there are only finitely many primes up to $d-1$, it is possible to choose the same $\tau=\tau(p)$ for all small primes $q\mid m, q\leq d-1$. Let $S_q$ contain those integers that lie in the residue class of $z_q \pmod{q^\tau}$.

\medskip

\noindent\emph{Case 2.} $d\leq q$. Note first that the congruence $p'(z)\equiv 0\pmod{q}$ has at most $d-1$ solutions. Let us choose exactly $d-1$ residue classes containing all these solutions, and let $S_q$ be the set consisting of all integers \emph{outside} of these $d-1$ residue classes$\pmod{q}$, i.e.~for every $z\in S_q$, 
$$p'(z)\not\equiv 0\pmod{q}.$$ 
The assumption that $S_q$ avoids  {\it exactly} $d-1$ residue classes will slightly simplify some formulas later.

Finally, we set 
$$Z:=Z(I, m)=I\cap \bigcap\limits_{q\mid m}S_q.$$
By construction, the furthermore part is clear.

We shall now prove that $Z\gg \frac{|I|}{(\log\log m)^{d-1}}$ and that for any $c$, $p(z)\equiv c\pmod {m}$ has at most $e^{O(\log m /\log\log m)}$  solutions in $Z$.


We start with showing that $Z$ is large.
Let 
$$J:=I\cap \bigcap\limits_{q\mid m,q\leq d-1}S_q.$$ Note that $J$ contains those elements of $I$ that lie in a specific residue class modulo $\prod\limits_{q\mid m, q\leq d-1} q^\tau$, thus 
\begin{equation}\label{eq-J}
	|J|\geq \frac{|I|}{d^{d\tau }}.
\end{equation}
 In other words, $J$ is an arithmetic progression with common difference  $\prod\limits_{q\mid m,q\leq d-1} q^\tau$, containing a positive portion of the elements of $I$.

As a next step, we estimate the size of $J\cap \bigcap\limits_{q\mid m, d\leq q} S_q=Z$ with the help of the inclusion-exclusion formula.

Let 
$$Q:=\{q:~d\leq q,~q\mid m\}$$ 
be the set of large prime divisors. For a subset $Q'\subseteq Q$, let $a(Q')$ be the number of those elements of $J$ that lie {\it outside} of $S_q$ for every $q\in Q'$:
$$a(Q'):=|J\setminus \bigcup\limits_{q\in Q'} S_q|.$$
Since $J$ is an arithmetic progression with common difference relatively prime to $\prod\limits_{q\in Q'}q$, any $\prod\limits_{q\in Q'}q$ consecutive elements of $J$ form a complete system of residues modulo $\prod\limits_{q\in Q'}q$. From such a system, by the definition of $S_q$ for $q\ge d$, exactly $\prod\limits_{q\in Q'}\frac{d-1}{q}$ proportion of the elements lie in $\bigcap\limits_{q\in Q'} \overline{S_q}$. Since  $J$ contains at least $\frac{|J|}{\prod\limits_{q\in Q'}q}-1$ pairwise disjoint such ``intervals'' and can be covered by at most $\frac{|J|}{\prod\limits_{q\in Q'}q}+1$ such intervals, we have
$$a(Q')=|J|\prod\limits_{q\in Q'}\frac{d-1}{q}+ b(Q'),$$
where 
$$|b(Q')|\le \prod\limits_{q\in Q'}q\cdot \prod\limits_{q\in Q'}\frac{d-1}{q}= (d-1)^{|Q'|}.$$

According to the inclusion-exclusion principle, we have
$$|Z|=\sum\limits_{Q'\subseteq Q}(-1)^{|Q'|}a(Q')=|J|\prod\limits_{q\mid m,d\leq q}\left(1-\frac{d-1}{q}\right)+ \sum\limits_{Q'\subseteq Q}(-1)^{|Q'|}b(Q').$$
Note that 
$$|\sum\limits_{Q'\subseteq Q}(-1)^{|Q'|}b(Q')|\leq \sum\limits_{Q'\subseteq Q}|b(Q')|\leq \sum\limits_{Q'\subseteq Q}(d-1)^{|Q'|} = d^{|Q|}\leq d^{\omega(m)}=e^{O(\frac{\log m}{\log\log m})},$$ 
where the last estimate follows from the fact that the number of distinct prime divisors satisfies $\omega(m)=O(\frac{\log m}{\log\log m})$. Using Mertens's estimate~\cite{Mer}, we obtain that
$$
|Z|\gg |J|\prod\limits_{q\mid m,d\leq q}\left(1-\frac{d-1}{q}\right)-e^{O(\frac{\log m}{\log\log m})}\gg \frac{|J|}{(\log\log m)^{d-1}}-e^{O(\frac{\log m}{\log\log m})}\gg\frac{|J|}{(\log\log m)^{d-1}}.
$$

Therefore, together with~\eqref{eq-J}, we have
$$|Z|\gg  \frac{|I|}{(\log\log m)^{d-1}}.$$

We are left to give an upper bound for the number of solutions to $p(z)\equiv c\pmod{m}$ with $z\in Z$. By considering instead $p(z)-c$, we may assume $c=0$ in what follows. Note also that we used the function $p'$ to construct the set $Z$  and so the constant term in $p$ did not play any role.

According to the Chinese remainder theorem, it suffices to give upper bounds for the number of solutions to 
$$p(z)\equiv 0\pmod {q^{\alpha}}$$ 
with $z\in S_q$ for every $q\mid m$, where $\alpha=\alpha(q)$ is the exponent of $q$ in the canonical form of $m$.

First, let $q\leq d-1$. We claim that if $z\in S_q$, then the congruence $p(z)\equiv 0\pmod{q^{\alpha}}$ has at most $q^{O(\sqrt{\alpha})}$ solutions. 
Let 
$$\beta:=\lceil \sqrt{\alpha} \rceil,$$ 
so $\beta^2\ge \alpha$ and it suffices to bound the number of solutions to $p(z)\equiv 0\pmod{q^{\beta^2}}$. 

Clearly, $p(z)\equiv 0 \pmod{q^\beta}$ has at most $q^\beta$ solutions. Note that if $p(z)\equiv 0 \pmod{q^{j\beta}}$, then $p(z)\equiv 0 \pmod{q^{(j-1)\beta}}$. We can write $$z=z_0+kq^{(j-1)\beta},$$
where $p(z_0)\equiv 0 \pmod{q^{(j-1)\beta}}$. Since $$p(z)=p(z_0+kq^{(j-1)\beta})\equiv p(z_0)+kq^{(j-1)\beta}\cdot p'(z_0)\pmod {q^{j\beta}},$$
if both $k$ and $k'$ satisfy the above congruence, then 
$$(k-k')\cdot  p'(z_0)\equiv 0\pmod{q^{\beta}}.$$
Using that $p'(z_0)\equiv p'(z)\not\equiv 0\pmod{q^\tau}$ due to the construction of $S_q$, we see that
$$k-k'\equiv 0\pmod{q^{\beta-\tau+1}}.$$ 
That is, the residue class $k\pmod{q^{\beta-\tau+1}}$ is uniquely determined. Hence, from each solution $z_0\pmod{q^{(j-1)\beta}}$ we get at most $q^{\tau-1}$ solutions $z\pmod{q^{j\beta}}$. Therefore, the number of solutions is at most $$q^{\beta+(\beta-1)(\tau-1)}=q^{O(\sqrt{\alpha})}.$$

Secondly, let $d\leq q$. We claim that if $z\in S_q$, then $p(z)\equiv 0\pmod{q^{\alpha}}$ has at most $d$ solutions with $z\in S_q$. We use induction on $\alpha$. Clearly the congruence $p(z)\equiv 0\pmod {q}$ has at most $d$ solutions. Now assume that $p(z)\equiv 0 \pmod {q^\gamma}$ has at most $d$ solutions for some $\gamma\ge 1$. If $p(z)\equiv 0\pmod {q^{\gamma+1}}$, then $z$ can be written as 
$$z=z_0+kq^\gamma,$$ 
where $p(z_0)\equiv 0\pmod{q^\gamma}$. Similarly, since 
$$p(z)\equiv p(z_0)+kq^\gamma\cdot  p'(z_0)\pmod{q^{\gamma+1}}$$ 
and $p'(z_0)\equiv p'(z)\not\equiv 0\pmod q$, the residue class $k\pmod {q}$ is uniquely determined. Therefore, from each solution of  $p(z)\equiv 0\pmod {q^{\gamma}}$ we get exactly one solution of $p(z)\equiv 0\pmod {p^{\gamma+1}}$. Hence, by induction, $p(z)\equiv 0 \pmod {q^{\alpha}}$ has at most $d$ solutions. 

We therefore obtain that the number of solutions to $p(z)\equiv 0\pmod{m}$ with $z\in Z$ is at most 
$$\left(\prod\limits_{q\mid m,q\leq d-1} q^{O(\sqrt{\alpha})}\right)\cdot \left(\prod\limits_{q\mid m,q\ge d} d\right)=e^{O(\sqrt{\log m})}\cdot d^{\omega(m)}=e^{O(\frac{\log m}{\log\log m})},$$
where the first estimate follows from the fact that $\alpha=O(\log m)$. This completes the proof of the lemma.
\end{proof}

We are now ready to prove Theorem~\ref{thm-numberofsol}. 
\subsection{Proof of Theorem~\ref{thm-numberofsol}}

We start with constructing the colouring for the second (easier) statement.

\subsubsection{Construction.} Let $n_2\leq n$ be the smallest positive integer such that $p(n_2)>2n$. Similarly, let $n_1$ be the smallest positive integer such that $p(n_1)>2(n_2-1)$. Recall that if $n$ is sufficiently large, then $p$ is strictly increasing on $[n_1,n]$ and $n_2\approx (2n/a_d)^{1/d}$, $n_1\approx (2n_2/a_d)^{1/d}$, and so $$n_1\approx (2/a_d)^{(d+1)/d^2}n^{1/d^2}.$$ 

Colour $[n_1-1]\cup [n_2,n]$ with colour 1 and $[n_1,n_2)$ with colour $-1$. Since $p(n_1)>2(n_2-1)$, any monochromatic solution must be in colour 1. Similarly, as $p(n_2)>2n$, if $x+y=p(z)$, then $z\in [n_1-1]$. Moreover, the minimality of $n_1$ implies that $p(z)\le 2(n_2-1)$, whence 
$$\min\{x,y\}\in [n_1-1].$$ 
Hence, the number of monochromatic solutions is at most $2n_1^2=O(n^{2/d^2})$.

\smallskip


We now continue with the lower bound on the number of monochromatic solutions. The linear case $d=1$ has already been proven by Robertson and Zeilberger~\cite{RZ98} and independently by Schoen~\cite{Sch99}. We assume now $d\ge 2$ and consider non-linear polynomials over $\Z$.

Recall the definition of a switch in Section~\ref{subsec-monotone}. We distinguish two types of switches. We say that a switch $k$ is \emph{isolated}, if the intervals $[k/2,k]$ and $[k+1,2k]$ are monochromatic (in different colours), otherwise, the switch is {\it non-isolated}. We will use different strategies to find monochromatic solutions for isolated and non-isolated switches.

\subsubsection{Isolated switches.} Consider a large isolated switch $k$. Without loss of generality we may assume that $[k+1,2k]$ is coloured $-1$ and $[0.5k,k]$ is coloured $+1$. 
Let
$$I:=[(1.51k/a_d)^{1/d},(1.59k/a_d)^{1/d}],$$
and denote by $I^+\subseteq I$ (resp. $I^-$) all elements of colour 1 (resp. $-1$) in $I$. By definition, for every $z\in I$, we have 
$$1.5k\le p(z)\le 1.6k.$$

\smallskip

\noindent{\it Case 1. $|I^+|=\Omega(k^{1/d})$.} Then for every $z\in I^+$ and every $x\in (0.8k,0.9k)$, we have $p(z)-x\in (0.6k,0.8k)$. Hence, the triple $\{x,y=p(z)-x,z\}$ is a monochromatic solution, yielding a total of $\Omega(k^{1+1/d})$ monochromatic solutions.

\smallskip

\noindent{\it Case 2. $|I^-|=\Omega(k^{1/d})$.} As $I^-\subseteq [0.4k]$, there are $\Omega(k^{1/d})$ many $x\in[0.4k]$ with colour $-1$. For every $z\in I^-$ and every $x\in [0.4k]$ with colour $-1$, $p(z)-x\in [1.1k, 1.6k]$. Therefore,
$\{x,y=p(z)-x,z\}$ is a monochromatic solution, yielding a total of $\Omega(k^{2/d})$ monochromatic solutions.

Therefore, the number of monochromatic solutions in either case is at least $\Omega(k^{2/d})$. 

\subsubsection{Non-isolated switches.} If $k$ is a non-isolated switch, then there is another switch in $[k/2,2k]$, thus assume that $k_1<k_2$ are two consecutive switches such that $k_2<2k_1$. 

Let $k\in\{k_1,k_2\}$ and assume without loss of generality that $\phi(k)=1$. Let again $m(k):=p(k+1)-p(k)$ and $R:=R_k\subseteq \Z_{m(k)}$ be the union of all non-monotone residue classes (recall the monotonicity defined in Section~\ref{subsec-monotone}). If $|R|\ge \frac{k}{e^{c(\log k/\log\log k)}}$, for some $c>0$ to be determined later, then by Lemma~\ref{lem-monotone}, we get at least 
$|R|\ge \frac{k}{e^{c(\log k/\log\log k)}}=k^{1-o(1)}$
monochromatic solutions. We may then assume that
 $$|R|\le \frac{k}{e^{c(\log k/\log\log k)}}$$ 
and so there are at least 
$m(k)-\frac{k}{e^{c(\log k/\log\log k)}}$ 
monotone residue classes. By definition, each  monotone class $j$ is such that $\phi\big|_{H_j}=\{-1,-1,\dots,-1,1,1,\dots,1\}$. 

Define
$$R':=\{j\in \Z_{m(k)}:~ p(k)-j\in R\}.$$
Let $A=A_k\subseteq \Z_{m(k)}\setminus (R\cup R')$ consist of all the monotone residue classes $j$ such that the colour is $-1$ within $H_j$ up to $p(k)/3$.  Observe crucially that

\medskip

\begin{itemize}
	\item[$(\dagger)$]\label{item-dagger} \emph{for any $j\notin R\cup R'$, if neither $j$ nor $p(k)-j\pmod{m(k)}$ belongs to $A$, then we get $\Omega(k)$ monochromatic solutions with $z=k$, $x\in H_j,y\in H_{p(k)-j}$.}
\end{itemize}

\medskip

\noindent Indeed, for every $x\in H_j\cap (p(k)/3,2p(k)/3)$ the triple $\{x,y=p(k)-x,z=k\}$ is a monochromatic solution.

We may then assume that for every $j\notin R\cup R'$, either $j$ or $p(k)-j\pmod{m(k)}$ belongs to $A$, and so 
$$|A|\geq |\Z_{m(k)}\setminus (R\cup R')|/2\ge m(k)/2-\frac{k}{e^{c(\log k/\log\log k)}}.$$

We claim that $A+A$ covers almost all the residue classes.
\begin{claim}\label{cl-goodclasses}
	$|A+A|\geq m(k)- \frac{6k}{e^{c(\log k/\log\log k)}}.$
\end{claim}
\begin{proof}
	By Kneser's theorem~\cite{Kneser}, $|A+A|\geq 2|A|-|H|$, where $H$ is the stabiliser of  $A+A$. It suffices to show that $|H|\le \frac{4k}{e^{c(\log k/\log\log k)}}$. If $A$ contains at least one element from more than half of the $H$-cosets, then $|A+H|>m(k)/2$ and by the pigeonhole principle $$\mathbb{Z}_{m(k)}=(A+H)+(A+H)=A+A,$$ 
	thus $|A+A|=m(k)$. We may then assume that $A$ contains elements from at most half of the $H$-cosets.
    
    Suppose the index of $H$, $m(k)/|H|$, is not even, then $|A|\leq m(k)/2-|H|/2$, thus $|H|\leq \frac{2k}{e^{c(\log k/\log\log k)}}$ as desired.
    
    We are left with the case that $m(k)/|H|$ is even and $A$ contains elements from exactly half of the $H$-cosets. In this case, $|A+H|=m(k)/2$. Consider a pair of elements $a$ and $b$ with $a+b\equiv p(k)\pmod{m(k)}$. Then there are at least 
    $|H|/2$ unordered pairs $(a',b')=(a+h,b-h)\in (a+H)\times (b+H)$ (it could be that $a+H=b+H$) satisfying 
    $$a'+b'=a+b\equiv p(k)\pmod{m(k)}.$$
    We may assume that $|H|> 4|R|$, otherwise $|H|\le \frac{4k}{e^{c(\log k/\log\log k)}}$ as desired. Then 
    $$|H|/2>|R\cup R'|,$$ 
    implying that one such pair $(a',b')=(a+h,b-h)$ consists of two monotone classes. Then by~$(\dagger)$, either $a+h$ or $b-h$ is in $A$. Consequently, either $a+h+H=a+H$ or $b-h+H=b+H$ is contained in $A+H$.
    
    We have observed that if $a+b\equiv p(k)\pmod{m(k)}$, then either $a+H$ or $b+H$ is contained in $A+H$. We can then pair up such $H$-cosets. As $2|p(1)p(2)$, either $p(k)$ or $p(k+1)=p(k)+m(k)$ is even, and so either $p(k)/2+H$ or $p(k+1)/2+H$ is paired up with itself. We then deduce that in fact $A+H$ contains more than half of the $H$-cosets, i.e.~$|A+H|>m(k)/2$, a contradiction. 
\end{proof}

We will use the following observation to find monochromatic solutions. If $z\in (0.1k,0.3k)$ is such that $\phi(z)=-\phi(k)=-1$ and $p(z)$ belongs to $A+A$ mod $m(k)$, say $p(z)\equiv a+a'\pmod{m(k)}$, then we can find $\Omega(k)$ monochromatic solutions to $x+y=p(z)$. Indeed, recall that for any $j\in A$, $H_j$ has an initial segment of colour $-1$ up to $p(k)/3$. For the above choice of $z$, we have 
$$((0.1)^d+o(1))p(k)\le p(z)\le p(k)/3.$$ 
Then for any $x\in H_a$ with $x<p(z)$, we get a monochromatic solution $\{x,y=p(z)-x,z\}$.

In the above observation, we require $z$ to have the opposite colour of $k$. To drop this requirement, we consider now 
$$z\in (0.1k_1,0.3k_1)\cap(0.1k_2,0.3k_2)=:I$$ 
such that $p(z)$ belongs to $A_{k_i}+A_{k_i}$ mod $m_i:=m(k_i)$ for each $i\in\{1,2\}$. Note that there exists $i^*\in \{1,2\}$ such that $\phi(k_{i^*})=-\phi(z)$ as $k_1,k_2$ are two consecutive switches having opposite colours. Then, for this choice of $z$, we can find $\Omega(k_{i^*})$ monochromatic solutions using $x,y$ from appropriate residue classes of $A_{i^*}$. 

Define for each $i\in\{1,2\}$,
$$X_{i}:=\mathbb{Z}_{m_i}\setminus (A_{k_i}+A_{k_i}),$$ 
we have by Claim~\ref{cl-goodclasses} that
$$|X_i|\leq \frac{6k_i}{e^{c(\log k_i/\log\log k_i)}}.$$

On the other hand, as $k_2<2k_1$, $I\supseteq (0.2k_1,0.3k_1)$ is an interval of length $\Theta(k_2)\ll k_2^{d-1}\ll m_1, m_2$. We can then apply Lemma~\ref{henseltype} for each $m\in \{m_1m_2, m_1, m_2\}$ to obtain sets $Z:=Z(I, m_1m_2)$, $Z_i:=Z(I, m_i)$, $i\in\{1,2\}$ such that
$$|Z|=\Omega\left(\frac{k_2}{(\log\log k_2)^{d-1}}\right), \quad \mbox{ and }\quad Z\subseteq Z_1\cap Z_2.$$
Since $Z\subseteq Z_i$ for each $i\in\{1,2\}$, at most 
$$(|X_1|+|X_2|)e^{O(\log k_2/\log\log k_2)}\le |Z|/2$$ 
elements of $Z$ are mapped to $X_1$ mod $m_1$ or to $X_2$ mod $m_2$, where the inequality follows from choosing $c$ large enough (but independent of $k_2$). 

Therefore, at least $|Z|/2$ elements of $Z\subseteq I$ are mapped out of $X_i$ mod $m_i$, for each $i\in\{1,2\}$. In other words, each such $z$ satisfies $p(z)\in A_{k_i}+A_{k_i}\pmod{m_i}$ for each $i\in\{1,2\}$. By the above observation, we get $\Omega(k)\cdot |Z|/2=\Omega\left(\frac{k^2}{(\log\log k)^{d-1}}\right)$ monochromatic solutions.

Hence, in the case of a non-isolated switch $k$, we get at least $k^{1-o(1)}$ monochromatic solutions. Note that in this case, the value of $z$ in the solutions we found can be as large as $k+1$. Since $p(z)\le p(k+1)\le n$ has to lie within the interval $[n]$, the non-isolated switch $k$ we consider should satisfy $k=O(n^{1/d})$.

\subsubsection{Putting things together.} To conclude the proof, take an interval 
$$J:=(c_1n^{1/d^2},c_2n^{1/d})$$ 
such that 
\begin{itemize}
\item $p(\cdot)$ is strictly increasing on $(c_1n^{1/d^2},\infty)$,
\item $p(c_1n^{1/d^2})<c_2n^{1/d}/10$,
\item $p(c_2n^{1/d})<n/10$.
\end{itemize} 
If $J$ is monochromatic, then we get  $\Omega(n^{1/d+1/d^2})$ monochromatic solutions by choosing $x=\Theta(n^{1/d})$ and $z=\Theta(n^{1/d^2})$. Otherwise, take a switch $k\in J$. This is either isolated or non-isolated, however, in both cases we get at least $$\min\{\Omega(k^{2/d}), k^{1-o(1)}\}=k^{2/d-o(1)}=n^{2/d^3-o(1)}$$ 
monochromatic solutions.

This completes the proof of Theorem~\ref{thm-numberofsol}.

\section{Concluding remarks}\label{sec-remark}
In this paper, we settle the Ramsey problem for $\{x,y,z:x+y=p(z)\}$ for all polynomials $p$ over $\Z$.

\subsection{Monochromatic solution in intervals}
The proof of Theorem~\ref{thm-numberofsol} yields also the following quantitative version.
\begin{theorem}\label{thm-interval}
	Let $p(z)=a_dz^d+\dots+a_1z+a_0\in\mathbb{Z}[z]$ with $d\ge 1$ and $a_d>0$ such that $2\mid p(1)p(2)$. Then for every 2-colouring of $[N,N^{d^3+o(1)}]$ with $N$ sufficiently large, there is a monochromatic solution to $x+y=p(z)$. Moreover, there exists some $c>0$ and a 2-colouring of $[N,cN^{d^2}]$ without monochromatic solutions of $x+y=p(z)$.
\end{theorem}
We believe that in both Theorems~\ref{thm-numberofsol} and~\ref{thm-interval}, the $2$-colourings constructed give the optimal exponents, i.e.~$2/d^2$ and $d^2$ respectively. We put forward as an open problem to close the gaps in the exponents.

\subsection{Further directions}
One further direction is to investigate other linear forms $a_1x_1+\ldots+a_kx_k$ instead of $x+y$.

Another further direction of study is to consider functions $f(z)$ beyond polynomials: $x+y=f(z)$.

Consider the exponential function $f(z)=2^z$. We define a 2-colouring of $\mathbb{N}$ recursively as follows. Let $\phi(1)=\phi(2)=1,\phi(3)=-1$ and for $k\geq 2$ and $x\in [2^k,2^{k+1})$ let $\phi(x)=-\phi(k+1)$. Assume that $x+y=2^z$. If $\max\{x,y\}\in [2^k,2^{k+1})$, then $2^k+1<x+y<2^{k+2}$, thus $z=k+1$. If $\max\{x,y\}\geq 4$, then $\max\{x,y\}$ and $z$ have different colours. If $\max\{x,y\}\leq 3$, then we get only one monochromatic solution, the trivial one: $x=y=z=1$.

Our approach can be applied to ``nice'' functions that grow sub-exponentially, up to $f(z)\sim e^{(\log z)^{2}}$.
\begin{theorem}\label{thm-functions}
	Let $f:\mathbb{R}_+\to \mathbb{R}_+$ be a monotone increasing convex differentiable function such that $f(\mathbb{N})\subseteq \mathbb{N}$ and the following conditions:
	\begin{itemize}
		\item either $f(n)$ is always even or $f(n)$ is even iff $n$ is odd (or iff $n$ is even)
		\item $2f'(k+1)\leq f(0.7k)$ 
	\end{itemize}
	Then for every $2$-colouring of $\mathbb{N}$ the equation $x+y=f(z)$ has infinitely many monochromatic solutions.
\end{theorem}

Note that every polynomial (satisfying the first necessary condition) satisfies these conditions (for large $z$). It would be interesting to determine the threshold, between $e^{(\log z)^{2}}$ and $2^z$, for the 2-Ramseyness of $x+y=f(z)$.

\end{document}